\documentclass[11pt,a4paper]{article}%
\usepackage{amsmath, amsfonts, amsthm, color,latexsym}
\usepackage{amscd}%
\usepackage{amsmath}%
\usepackage{amsfonts}%
\usepackage{amssymb}%
\usepackage{graphicx}
\providecommand{\U}[1]{\protect\rule{.1in}{.1in}}
\newtheorem{theorem}{Theorem}[section]
\newtheorem{proposition}[theorem]{Proposition}
\newtheorem{corollary}[theorem]{Corollary}
\newtheorem{example}[theorem]{Example}
\newtheorem{examples}[theorem]{Examples}
\newtheorem{remark}[theorem]{Remark}

\newtheorem{final remark}[theorem]{Final Remark}
\newtheorem{definition}[theorem]{Definition}
\begin{document}

\title{\textsc{{On the representation of multi-ideals by
tensor norms}}}
\author{Geraldo Botelho\thanks{Supported by CNPq Grant 306981/2008-4.}~, Erhan \c Caliskan and Daniel Pellegrino\thanks{Supported by INCT-Matem\'{a}tica, CNPq Grants 620108/2008-8 (Ed. Casadinho), 471686/2008-5 (Ed. Universal), 308084/2006-3 and PROCAD-NF Capes.\newline
}}
\date{}
\maketitle
\begin{abstract} A tensor norm $\beta= (\beta_{n})_{n=1}^{\infty}$ is smooth if the natural correspondence
$(E_{1} \otimes\cdots\otimes E_{n} \otimes\mathbb{K},\beta_{n+1}) \longleftrightarrow
(E_{1} \otimes\cdots\otimes E_{n} ,\beta_{n})$ is always an isometric isomorphism. In this paper we study the representation of multi-ideals and of ideals of multilinear forms by smooth tensor norms.
\end{abstract}

\section{Introduction and notation}
The idea of describing the dual of a topological tensor product by means of a special class of bilinear mappings goes back to Grothendieck's celebrated R\'esum\'e \cite{resume}. For example, in his seminal work Grothendieck showed that to linear functionals on the injective tensor product correspond integral bilinear forms. With the emergence of the theory of ideals of multilinear mappings (multi-ideals) between Banach spaces, several instances of this kind of correspondence have appeared. For example, Matos in \cite{complutense} constructs a tensor norm such that to linear operators on the tensor product which are continuous with respect to this norm correspond exactly the class of nuclear multilinear mappings.\\
\indent In the theory of multi-ideals, the possibility of moving smoothly from spaces of $(n+1)$-linear mappings down to spaces of $n$-linear mappings turned out to be an important property (see, e.g., \cite{indagationes, studia, muro}). In this note we study
tensor norms in which this transition is smooth, as well as multi-ideals that correspond to such smooth tensor norms. To be more precise we need some definitions:

\begin{definition}\rm An {\it $n$-tensor norm} $\beta_{n}$ assigns to every $n$-tuple of normed spaces
$E_{1},\ldots,E_{n}$ a reasonable crossnorm $\beta_{n}(\cdot)$ on the full
$n$-fold tensor product $E_{1}\otimes\cdots\otimes E_{n}$ which satisfies the
metric mapping property. The resulting normed space is denoted by
$(E_{1}\otimes\cdots\otimes E_{n},\beta_{n})$. A {\it tensor norm} is a sequence
$\beta=(\beta_{n})_{n=1}^{\infty}$ where each $\beta_{n}$ is an $n$-tensor
norm.
\end{definition}

Let $E,E_1, \ldots, E_n,F$ be (real or complex) Banach spaces. By ${\cal L}(E;F)$ we denote the space of bounded linear operators from $E$ to $F$ endowed with the usual operator norm. When $F$ is the scalar field we simply write $E'$. By ${\cal L}(E_1, \ldots, E_n;F)$ we mean the space of continuous $n$-linear mappings from $E_1 \times \cdots \times E_n$ to $F$ endowed with the usual sup norm. When $E_1 = \cdots = E_n = E$ we write ${\cal L}(^nE;F)$.

\begin{definition}\rm \label{Definition 1.3} \rm An {\it ideal of multilinear mappings} (or {\it multi-ideal}) $\cal M$ is a
subclass of the class of all continuous
 multilinear mappings between Banach spaces such that for a positive integer $n$, Banach spaces $E_1, \ldots, E_n$ and $F$, the components ${\cal M}
 (E_1, \ldots, E_n;F) := {\cal L}(E_1, \ldots, E_n;F) \cap {\cal M}$ satisfy:
 \\
\noindent (i) ${\cal M}(E_1, \ldots, E_n;F)$ is a linear subspace of
${\cal L}(E_1, \ldots, E_n;F)$ which contains the $n$-linear mappings of finite type.\\
(ii) The ideal property: if $A \in {\cal M}(E_1, \ldots, E_n;F)$,
$u_j \in {\cal L}(G_j;E_j)$ for $j = 1, \ldots, n$ and $t \in
{\cal L}(F;H)$, then $t \circ A \circ (u_1, \ldots, u_n)$ is in
${\cal M}(G_1, \ldots, G_n;H)$.

\medskip

\noindent Moreover, there is a function $\|\cdot\|_{\cal M} \colon {\cal
M} \longrightarrow \mathbb{R}^+$ satisfying

\medskip

\noindent (i') $\|\cdot\|_{\cal M}$ restricted to ${\cal M}(E_1, \ldots,
E_n;F)$ is a norm for all
 Banach spaces $E_1, \ldots, E_n$ and $F$, which makes ${\cal M}(E_1, \ldots,
E_n;F)$ a Banach space. \\
\noindent (ii') $\|A \colon \mathbb{K}^n \longrightarrow \mathbb{K} : A(\lambda_1,\ldots,
 \lambda_n)
 = \lambda_1 \cdots \lambda_n \|_{\cal M} = 1$ for all $n$,\\
\noindent (iii') If $A \in {\cal M}(E_1, \ldots, E_n;F)$, $u_j \in
{\cal L}(G_j;E_j)$ for $j = 1, \ldots, n$ and $t \in {\cal
L}(F;H)$, then $\|t \circ A \circ (u_1, \ldots, u_n)\|_{\cal M}
\leq \|t\| \|A\|_{\cal M} \|u_1\|\cdots \|u_n\|$.

Of course the Banach spaces considered in this definition are all over the same fixed scalar field $\mathbb{K} = \mathbb{R}$ or $\mathbb{C}$. We
define
\[
\mathcal{M}^{\mathbb{K}}:= \{\mathcal{M}(E_{1}, \ldots, E_{n}; \mathbb{K}): n
\in\mathbb{N} {\rm~and~} E_{1}, \ldots, E_{n} {\rm~are ~Banach~spaces}\},
\]
and say that $\mathcal{M}^{\mathbb{K}}$ is an {\it ideal of multilinear forms}.
\end{definition}

\begin{definition}\rm
\label{definition}We say that a tensor norm $\beta= (\beta_{n}%
)_{n=1}^{\infty}$ {\it represents} the multi-ideal $\mathcal{M}$ - or
$\mathcal{M}$ is $\beta$-represented - if $\mathcal{M}(E_{1}, \ldots,
E_{n};F^{\prime})$ is {isometrically} isomorphic to $(E_{1}
\otimes\cdots\otimes E_{n}\otimes F, \beta_{n+1})^{\prime}$ through
the canonical mapping
\[
\varphi\colon\mathcal{M}(E_{1}, \ldots, E_{n};F^{\prime}) \longrightarrow
(E_{1} \otimes\cdots\otimes E_{n}\otimes F, \beta_{n+1})^{\prime}%
\]
\[
T \mapsto\varphi(T)(x_{1} \otimes\cdots\otimes x_{n} \otimes y) = T(x_{1},
\ldots, x_{n})(y),
\]
for every $n$ and every Banach spaces $E_{1}, \ldots, E_{n}, F$. \\
\indent The ideal of multilinear forms $\mathcal{M}^{\mathbb{K}}$ is {\it represented} by $\beta$ - or $\mathcal{M}^{\mathbb{K}%
}$ is $\beta$-represented - if the condition above holds for all components $\mathcal{M}(E_{1},
\ldots, E_{n}; \mathbb{K})$ of $\mathcal{M}^{\mathbb{K}}$.
\end{definition}

\indent In \cite[Theorem 4.5]{hf} it is proved that a multi-ideal
$\mathcal{M}$ is maximal if and only if $\mathcal{M}$ is represented by some
(finitely generated) tensor norm. It is well known (see, e.g., \cite[Exercise
12.1]{defantfloret}) that for every normed space $E$ and every 2-tensor norm
$\alpha$, $(E \otimes\mathbb{K}, \alpha)$ is isometrically isomorphic to $E$
via the correspondence $x \otimes\lambda\longleftrightarrow\lambda x$. Given a
tensor norm $\beta= (\beta_{n})_{n=1}^{\infty}$, this property can be
rewritten as $(E \otimes\mathbb{K}, \beta_{2}) = (E, \beta_{1})$ for every
$E$. As we will see along the paper, this property is no longer valid for larger $n$, that is, it is not always true that
$(E_1\otimes\cdots \otimes E_n, \beta_{n})$ is canonically isomorphic to
$(E_1\otimes\cdots \otimes E_n \otimes\mathbb{K}, \beta_{n+1}) $ for
every $n\geq2$. This phenomenon motivates the following definition:

\begin{definition}\rm
\label{def}\textrm{A tensor norm $\beta= (\beta_{n})_{n=1}^{\infty}$ is said
to be \textit{smooth} if, regardless of the natural $n$ and the normed spaces
$E_{1}, \ldots, E_{n}$, the natural map
\[
\psi\colon(E_{1} \otimes\cdots\otimes E_{n} \otimes\mathbb{K},\beta_{n+1})
\longrightarrow(E_{1} \otimes\cdots\otimes E_{n} ,\beta_{n}),
\]
\[
\psi(x_{1} \otimes\cdots\otimes x_{n} \otimes\lambda) =\lambda(x_{1}
\otimes\cdots\otimes x_{n} )
\]
is an isometric isomorphism. }
\end{definition}

In this paper we are concerned with the representation of multi-ideals by smooth tensor norms. The conclusion of our results/examples is that: (i) multi-ideals are rarely represented by smooth tensor norms, (ii) ideals of multilinear forms are more often represented by smooth tensor norms, (iii) the representation of an ideal of multilinear forms by a smooth tensor norm yields the representation of some of its vector-valued components by the same smooth tensor norm.

\section{Vector-valued case}
The aim of this section is to show that multi-ideals, including their vector-valued components, are rarely represented by smooth tensor norms. We start with the two first obvious examples.

\begin{examples}\rm
\label{example}(a) The projective tensor norm $\pi$ is smooth: given
$z = \sum_{j=1}^{m} x_{j}^{(1)}\otimes\cdots\otimes x_{j}^{(n)}\otimes
\lambda_{j} \in E_{1}\otimes\cdots\otimes E_{n}\otimes\mathbb{K}$, letting
$\bar z = \sum_{j=1}^{m} \lambda_{j}x_{j}^{(1)}\otimes\cdots\otimes
x_{j}^{(n)}\in E_{1}\otimes\cdots\otimes E_{n} $, it is easy to check that
$\pi_{n+1}(z) = \pi_{n}(\bar z)$. So $\psi\colon(E_{1} \otimes\cdots\otimes
E_{n} \otimes\mathbb{K},\pi_{n+1}) \longrightarrow(E_{1} \otimes\cdots\otimes
E_{n} ,\pi_{n})$ is an isometric isomorphism. So the multi-ideal $\mathcal{L}$
of all continuous multilinear mappings between Banach spaces, which is
obviously $\pi$-represented (see \cite[Proposition A.3.7]{dales}), is
represented by a smooth tensor norm. \\

\noindent (b) The injective tensor norm $\varepsilon$ is smooth.
Indeed, for $z$ and $\bar{z}$ as above,
\begin{align*}
\varepsilon_{n+1}(z)  & =\sup_{\varphi_{l}\in B_{E_{l}^{\prime}},\varphi\in
B_{\mathbb{K}^{\prime}}}\left\vert \sum_{j=1}^{m}\varphi_{1}(x_{j}%
^{(1)})\cdots\varphi_{n}(x_{j}^{(n)})\varphi(\lambda_{j})\right\vert \\
& =\sup_{\varphi_{l}\in B_{E_{l}^{\prime}},\varphi\in B_{\mathbb{K}^{\prime}}%
}\left\vert \sum_{j=1}^{m}\varphi_{1}(\lambda_{j}x_{j}^{(1)})\cdots\varphi
_{n}(x_{j}^{(n)})\varphi(1)\right\vert \\
& =\sup_{\varphi_{l}\in B_{E_{l}^{\prime}},\varphi\in B_{\mathbb{K}^{\prime}}%
}|\varphi(1)|\left\vert \sum_{j=1}^{m}\varphi_{1}(\lambda_{j}x_{j}%
^{(1)})\cdots\varphi_{n}(x_{j}^{(n)})\right\vert \\
& =\sup_{\varphi_{l}\in B_{E_{l}^{\prime}}}\left\vert \sum_{j=1}^{m}%
\varphi_{1}(\lambda_{j}x_{j}^{(1)})\cdots\varphi_{n}(x_{j}^{(n)})\right\vert
=\varepsilon_{n}(\widetilde{z}),
\end{align*}
proving that $\psi\colon(E_{1}\otimes\cdots\otimes E_{n}\otimes\mathbb{K}%
,\varepsilon_{n+1})\longrightarrow(E_{1}\otimes\cdots\otimes E_{n}%
,\varepsilon_{n})$ is an isometric isomorphism. It is known since Grothendieck's R\'esum\'e that the
multi-ideal $\mathcal{L}_{\mathcal{I}}$ of integral multilinear mappings is $\varepsilon$-represented (the scalar-valued bilinear case can be found in \cite[Theorem 1.1.21]{diestelresume}). We give the details for the sake of completeness. Denoting by
$\mathcal{I}(E;F)$ the space of integral linear operators from $E$ to $F$ and
by $E_{1}\widehat{\otimes}_{\varepsilon}\cdots\widehat{\otimes}_{\varepsilon
}E_{n}$ the completion of $(E_{1}\otimes\cdots\otimes E_{n},\varepsilon_{n})$
we have that
\[
\mathcal{L}_{\mathcal{I}}(E_{1},\ldots,E_{n};F^{\prime})\overset{(1)}%
{=}\mathcal{I}(E_{1}\widehat{\otimes}_{\varepsilon}\cdots\widehat{\otimes
}_{\varepsilon}E_{n};F^{\prime})\overset{(2)}{=}\mathcal{L}_{\mathcal{I}%
}(E_{1}\widehat{\otimes}_{\varepsilon}\cdots\widehat{\otimes}_{\varepsilon
}E_{n},F;\mathbb{K})
\]%
\[
\overset{(1)}{=}\mathcal{I}(E_{1}\widehat{\otimes}_{\varepsilon}\cdots
\widehat{\otimes}_{\varepsilon}E_{n}\widehat{\otimes}_{\varepsilon
}F;\mathbb{K})\overset{(3)}{=}(E_{1}\widehat{\otimes}_{\varepsilon}%
\cdots\widehat{\otimes}_{\varepsilon}E_{n}\widehat{\otimes}_{\varepsilon
}F)^{\prime}\overset{(4)}{=}(E_{1}\otimes\cdots\otimes E_{n}\otimes
F,\varepsilon_{n+1})^{\prime}%
\]
(1) See \cite[Proposition 2.2]{cdg}.\newline(2) See \cite[Proposition 3.22]{Ryan}.\newline(3) This is trivial
because $\mathcal{I}$ is an operator ideal.\newline(4) Since $E_{1}%
\widehat{\otimes}_{\varepsilon}\cdots\widehat{\otimes}_{\varepsilon}%
E_{n}\widehat{\otimes}_{\varepsilon}F$ is the completion of the normed space
$(E_{1}\otimes\cdots\otimes E_{n}\otimes F,\varepsilon_{n+1})$, we can
identify $(E_{1}\otimes\cdots\otimes E_{n}\otimes F,\varepsilon_{n+1}%
)^{\prime}$ with $(E_{1}\widehat{\otimes}_{\varepsilon}\cdots\widehat{\otimes
}_{\varepsilon}E_{n}\widehat{\otimes}_{\varepsilon}F)^{\prime}$ through the
well known isometry $\varphi\longrightarrow\overline{\varphi}$, where
$\overline{\varphi}$ is the unique extension of $\varphi$.\newline%
\indent Since all identifications above are done by the corresponding standard
mappings, their composition coincides with our mapping $\varphi$, which proves
that $\mathcal{L}_{\mathcal{I}}$ is $\varepsilon$-represented, hence
represented by a smooth tensor norm.
\end{examples}

\begin{definition}\rm Let a tensor norm
$\beta= (\beta_{n})_{n=1}^{\infty}$ be given. Define $\mathcal{L}_{\beta
}(E_{1}, \ldots, E_{n};F)$ as those multilinear mappings $A \in\mathcal{L}%
(E_{1}, \ldots, E_{n};F)$ whose linearizations
\[
A_{L} \colon(E_{1} \otimes\cdots\otimes E_{n}, \beta_{n}) \longrightarrow F
\]
are continuous endowed with the norm $$\|A\|_{{\cal L}_{\beta}} :=
\|A_L \colon (E_1 \otimes \cdots \otimes E_n, \beta_n)
\longrightarrow F \|. $$
\end{definition}
It is easy to see that
$\mathcal{L}_{\beta}$ is a multi-ideal. The question
of whether or not ${\cal L}_{\beta}$ is $\beta$-represented is quite
natural. We shall treat it later.\newline\indent Next we define a
property which is closely related to property (B) of \cite{studia}:

\begin{definition}\rm
\textrm{Given $A \in\mathcal{L}(E_{1}, \ldots, E_{n}, \mathbb{K};F)$, define
$A1 \in\mathcal{L}(E_{1}, \ldots, E_{n};F)$ by $A1(x_{1}, \ldots, x_{n}) =
A(x_{1}, \ldots, x_{n},1)$. We say that a multi-ideal $\mathcal{M}$ has
property [B] if
\[
A \in\mathcal{M}(E_{1}, \ldots, E_{n}, \mathbb{K};F) \Longleftrightarrow A1
\in\mathcal{M}(E_{1}, \ldots, E_{n};F)
\]
{and in this case $\|A\|_{\cal M} = \|A1\|_{\cal M}$}, for every
$n$, $E_{1}, \ldots, E_{n},F$ and $A \in\mathcal{L}(E_{1}, \ldots,
E_{n};F)$. }
\end{definition}

\begin{proposition}
\label{prop} A tensor norm $\beta$ is smooth if and only if its corresponding
multi-ideal $\mathcal{L}_{\beta}$ has property [B].
\end{proposition}

\begin{proof} Assume that $\beta = (\beta_n)_{n=1}^\infty$ is a smooth tensor norm. Given
$A \in {\cal L}(E_1, \ldots, E_n,\mathbb{K};F)$, consider the chain
$$E_1 \otimes \cdots \otimes E_n \otimes \mathbb{K} \stackrel{\psi}{\longrightarrow} E_1 \otimes \cdots \otimes E_n
\stackrel{{\psi}^{-1}}{\longrightarrow} E_1 \otimes \cdots \otimes
E_n \otimes \mathbb{K} \stackrel{A_L}{\longrightarrow} F $$
It is
not difficult to see that $A_L = (A1)_L \circ \psi$ and $(A1)_L =
A_L \circ \psi^{-1}$. Since $\psi$ and $\psi^{-1}$ are continuous as
$\beta$ is smooth, it follows that $A_L$ is continuous if and only
if $(A1)_L$ is continuous, that is $A \in {\cal L}_\beta$ if and
only if $A_1 \in {\cal L}_\beta$. {In this case $\|A\|_{{\cal
L}_\beta} = \|A1\|_{{\cal L}_\beta}$ because $\psi$ and $\psi^{-1}$
are isometric
isomorphisms}, proving that ${\cal L}_\beta$ has property [B].\\
\indent Conversely, assume that ${\cal L}_\beta$ has property [B]. Given $E_1, \ldots, E_n$, consider $A
\colon E_1 \times \cdots \times E_n \times \mathbb{K} \longrightarrow (E_1 \otimes \cdots \otimes E_n,
\beta_n)$ defined by $A(x_1, \ldots, x_n,\lambda) = \lambda(x_1 \otimes \cdots \otimes x_n)$. Then
$$(A1)_L(x_1 \otimes \cdots \otimes x_n) = A1(x_1,
\ldots, x_n) = A(x_1, \ldots, x_n,1)= x_1 \otimes \cdots \otimes
x_n,$$ showing that $(A1)_L$ is the identity operator on $E_1
\otimes \cdots \otimes E_n$, hence {an isometric isomorphism} when
this space is endowed with $\beta_n$ on both sides. It follows that
$A1 \in {\cal L}_\beta(E_1, \ldots, E_n; (E_1 \otimes \cdots \otimes
E_n, \beta_n))$. So $A \in {\cal L}_\beta(E_1, \ldots, E_n,
\mathbb{K}; (E_1 \otimes \cdots \otimes E_n, \beta_n))$ {and
$$\|A\|_{{\cal L}_\beta} = \|A1\|_{{\cal L}_\beta} = \|(A1)_L \colon
(E_1 \otimes \cdots \otimes E_n, \beta_n) \longrightarrow (E_1
\otimes \cdots \otimes E_n, \beta_n) \| =1$$} as ${\cal L}_\beta$
has property [B]. Therefore
$$A_L \colon (E_1 \otimes \cdots \otimes E_n\otimes \mathbb{K},
\beta_{n+1}) \longrightarrow (E_1 \otimes \cdots \otimes E_n,
\beta_n) $$ is continuous. But $A_L = \psi$, so $\psi$ is
continuous and $$1 = \|A\|_{{\cal L}_\beta} = \|A_L \colon (E_1
\otimes \cdots \otimes E_n\otimes \mathbb{K}, \beta_{n+1})
\longrightarrow (E_1 \otimes \cdots \otimes E_n, \beta_n) \|. $$
Moreover, $\beta_n(\psi(z)) = \beta_n(A_L(z)) \leq \beta_{n+1}(z)$
for every $z \in E_1 \otimes \cdots \otimes E_n\otimes
\mathbb{K}$.\\
\indent Consider now $C \colon E_1 \times \cdots \times E_n \times
\mathbb{K} \longrightarrow (E_1 \otimes \cdots \otimes E_n\otimes
\mathbb{K}, \beta_{n+1})$ given by $C(x_1, \ldots, x_n, \lambda) =
x_1 \otimes \cdots \otimes x_n \otimes \lambda$. It is clear that
$C_L$ is the identity operator on $E_1 \otimes \cdots \otimes
E_n\otimes \mathbb{K}$, hence {an isometric isomorphism} when this
space is endowed with $\beta_{n+1}$ on both sides. Hence $C \in
{\cal L}_\beta(E_1, \ldots, E_n,\mathbb{K}; (E_1 \otimes \cdots
\otimes E_n\otimes \mathbb{K}, \beta_{n+1}))$. So $C1 \in {\cal
L}_\beta(E_1, \ldots, E_n; (E_1 \otimes \cdots \otimes E_n\otimes
\mathbb{K}, \beta_{n+1}))$ {and $$\|C1\|_{{\cal L}_\beta} =
\|C\|_{{\cal L}_\beta} = \|C_L \colon (E_1 \otimes \cdots \otimes
E_n\otimes \mathbb{K}, \beta_{n+1}) \longrightarrow (E_1 \otimes
\cdots \otimes E_n\otimes \mathbb{K}, \beta_{n+1}) \| =1$$} as
${\cal L}_\beta$ has property [B]. Therefore
$$(C1)_L \colon (E_1 \otimes \cdots \otimes E_n,
\beta_{n}) \longrightarrow (E_1 \otimes \cdots \otimes E_n\otimes
\mathbb{K}, \beta_{n+1}) $$ is continuous. But $(C1)_L = \psi^{-1}$,
so $\psi^{-1}$ is continuous  {and $$1 = \|C1\|_{{\cal L}_\beta} =
\|(C1)_L \colon (E_1 \otimes \cdots \otimes E_n, \beta_{n})
\longrightarrow (E_1 \otimes \cdots \otimes E_n\times \mathbb{K},
\beta_{n+1}) \|. $$ Moreover, $\beta_{n+1}(\psi^{-1}(w)) =
\beta_{n+1}((C1)_L(w)) \leq \beta_{n}(w)$ for every $w \in E_1
\otimes \cdots \otimes E_n$. Making $w = \psi(z)$ we obtain
$$\beta_n(\psi(z)) \leq \beta_{n+1}(z) \leq  \beta_n(\psi(z))$$
for every $z \in E_1 \otimes \cdots \otimes E_n\otimes \mathbb{K}$, proving that $\psi$ is an isometric
isomorphism, that is, $\beta$ is smooth.}
\end{proof}

Now we turn our attention to multi-ideals that can be represented by smooth
tensor norms.




\begin{theorem}
\label{theorem} {\rm(Uniqueness of the representation)}%
{ The tensor norm that represents a given multi-ideal, if any, is unique.}
\end{theorem}

\begin{proof}{ Let $\cal M$ be a multi-ideal that is represented by the tensor norms
$\beta = (\beta_n)_{n=1}^\infty$ and $\gamma = (\gamma_n)_{n=1}^\infty$. Let $E_1, \ldots, E_{n-1}, E_n$ be
given. We have that the corresponding operators $\varphi_\beta \colon {\cal M}(E_1, \ldots,
E_{n-1};E_n^\prime) \longrightarrow (E_1 \otimes\cdots \otimes E_n,
\beta_n)^\prime$ and $\varphi_{\gamma} \colon {\cal M}(E_1, \ldots, E_{n-1};E_n^\prime) \longrightarrow
(E_1 \otimes\cdots \otimes
E_n,
\gamma_n)^\prime$ are isometric isomorphisms. So the composition $\varphi_\beta \circ (\varphi_{\gamma})^{-1}$,
which is clearly the formal identity, is an isometric isomorphism from $(E_1 \otimes\cdots \otimes
E_n,
\gamma_n)^\prime$ to $(E_1 \otimes\cdots \otimes
E_n,
\beta_n)^\prime$. By the Hahn-Banach theorem it follows that $ \gamma_n = \beta_n$ on $E_1
\otimes \cdots \otimes E_n$.}
\end{proof}


Given a maximal multi-ideal $\mathcal{M}$ (for the definition see
\cite{hf}), we have already mentioned that \cite[Theorem 4.5]{hf}
assures the existence of a tensor norm that represents
$\mathcal{M}$. Let us denote such tensor norm, {which is unique by
Theorem \ref{theorem}}, by $\beta ^{\mathcal{M}}$. Combining
Proposition \ref{prop} and {Theorem \ref{theorem}} we have:

\begin{proposition}
\label{max}%
{The following are equivalent for a maximal multi-ideal $\cal M$:\\
{\rm (a)} $\cal M$ is represented by a smooth tensor norm.\\
{\rm (b)} $\beta^{\cal M}$ is smooth.\\
{\rm (c)} ${\cal L}_{\beta^{\cal M}}$ has property [B].\\
In particular, if $\cal M$ has property [B] and
${\cal M} = {\cal
L}_{\beta^{\cal M}}$ then $\cal M$ is represented by a smooth tensor norm.}
\end{proposition}

This result impels us to study the equality $\mathcal{M} = \mathcal{L}%
_{\beta^{\mathcal{M}}}$. As to the projective norm, by Example
\ref{example}(a) and {Theorem \ref{theorem}} we know that
$\beta^{\mathcal{L}} = \pi$, so $\mathcal{L}_{\beta^{\mathcal{L}}} =
\mathcal{L}_{\pi}= \mathcal{L}$. We treat this question together with the question of whether or
not ${\cal L}_{\beta}$ is $\beta$-represented.

\begin{proposition}\label{injec}
{ The multi-ideal ${\cal L}_{\varepsilon}$ is not
$\varepsilon-$represented.}
\end{proposition}

\begin{proof} Assume for a while that ${\cal L}_{\varepsilon}$ is $\varepsilon-$represented. On the one hand,
by Example \ref{example}(b) we know that ${\cal L}_{\cal I}$ is $\varepsilon-$represented, so it follows
easily that ${\cal L}_{\varepsilon}(E_1, \ldots, E_n;F') = {\cal L}_{\cal I}(E_1, \ldots, E_n;F')$ for every
$n$ and $E_1, \ldots, E_n,F$. 
On the other hand, by
Example \ref{example}(b) and {Theorem \ref{theorem}} we know that
$\beta^{\mathcal{L}_{\mathcal{I}}} = \varepsilon$ and from
\cite[Proposition 2.2]{cdg} we have $\mathcal{L}_{\mathcal{I}}
\subseteq\mathcal{L}_{\varepsilon}$, hence
$\mathcal{L}_{\mathcal{I}}
\subseteq\mathcal{L}_{\varepsilon}= \mathcal{L}_{\beta^{\mathcal{L}%
_{\mathcal{I}}}}$. As to the converse inclusion, let $n \geq2$ and $F$ be an
infinite dimensional Banach space. Assume for a while that $\mathcal{L}%
(\hat\otimes_{n,\varepsilon} c_{0};F) = \mathcal{I}(\hat\otimes_{n,\varepsilon
} c_{0};F)$. As integral linear opertors are absolutely summing
\cite[Proposition 5.5]{djt}, we have that every continuous linear operator
from $\hat\otimes_{n,\varepsilon} c_{0}$ to $F$ is absolutely summing. We know
that $\hat\otimes_{n,\varepsilon} c_{0}$ has unconditional basis because
$c_{0} = \hat\otimes_{n,\varepsilon} c_{0}$, so by a result due to
Lindenstrauss-Pe{\l }czy\'nski \cite[Theorem 4.2]{lp} it follows that
$\hat\otimes_{n,\varepsilon} c_{0} = c_{0}$ is isomorphic to some $\ell
_{1}(\Gamma)$, but this is absurd. Therefore there exists a non-integral
operator $u \in\mathcal{L}(\hat\otimes_{n,\varepsilon} c_{0};F)$. Define
\[
A \colon c_{0} \times\cdots\times c_{0} \longrightarrow F~,~A(x_{1}, \ldots,
x_{n}) = u(x_{1} \otimes\cdots\otimes x_{n}).
\]
So $A_{L} = u$ is $\varepsilon$-continuous, hence $A \in\mathcal{L}%
_{\varepsilon}(^{n} c_{0};F)$, but $A$ fails to be integral because its
linearization $A_{L} = u$ fails to be integral in the injective norm
$\varepsilon$. Hence $\mathcal{L}_{\beta^{\mathcal{L}_{\mathcal{I}}}}(^{n}
c_{0};F) = \mathcal{L}_{\varepsilon}(^{n} c_{0};F) \neq
\mathcal{L}_{\mathcal{I}}(^{n} c_{0};F)$ - a contradiction.
\end{proof}

Later, in Proposition \ref{propos}, we shall go quite further. For the moment, the proof above shows, in particular, that $\mathcal{L}_{\beta^{\mathcal{L}_{\mathcal{I}}}}\neq
\mathcal{L}_{\mathcal{I}}$. Corollary \ref{coro} shall provide another example of the inequality $\mathcal{M}
\neq\mathcal{L}_{\beta^{\mathcal{M}}}$.



Now we proceed to present some multi-ideals that are represented by tensor norms but not by smooth tensor norms.
Before we give a general criteria:

\begin{proposition}
Let $\mathcal{M}$ be a multi-ideal such that $\mathcal{M}(E_{1},
\ldots, E_{n},F;\mathbb{K}) = \mathcal{L}(E_{1}, \ldots,
E_{n},F;\mathbb{K})$ and $\mathcal{M}(E_{1}, \ldots,
E_{n};F^{\prime}) \neq\mathcal{L}(E_{1}, \ldots, E_{n};F^{\prime}) $
for some Banach spaces $E_{1}, \ldots, E_{n}, F$ and some positive
integer $n$.
Then there is no smooth tensor norm that represents
$\mathcal{M}$.\label{p:111}
\end{proposition}
\begin{proof}Assume  that there is a smooth tensor norm $\beta =
(\beta_n)_{n=1}^\infty$ that represents  ${\cal M}$. Since $\beta$ is
smooth, the adjoint $\psi^*$ of $\psi$ is an isometric isomorphism from $(E_1
\otimes \cdots \otimes E_n\otimes F, \beta_{n+1})^\prime$ to $(E_1 \otimes \cdots
\otimes E_n \otimes F\otimes \mathbb{K}, \beta_{n+2})^\prime$.
We have the following chain of isomorphisms:
$$(E_1 \otimes \cdots \otimes
E_n\otimes F, \beta_{n+1})^\prime \stackrel{\psi^*}{\longleftrightarrow} (E_1
\otimes \cdots \otimes E_n \otimes F\otimes \mathbb{K}, \beta_{n+2})^\prime
$$
$$\stackrel{\varphi}{\longleftrightarrow}  {\cal M}(E_1, \ldots, E_n, F;\mathbb{K}^\prime) \longleftrightarrow
{\cal M}(E_1, \ldots, E_n, F;\mathbb{K}) $$
$$\stackrel{id}{\longleftrightarrow} {\cal L}(E_1, \ldots, E_n, F;\mathbb{K})\longleftrightarrow (E_1 \otimes \cdots
\otimes E_n\otimes F,\pi_{n+1})^\prime, $$ where the non-indicated mappings are the
canonical ones. It follows that the identity operator is an
algebraic isomorphism between $(E_1 \otimes \cdots \otimes E_n\otimes F,
\beta_{n+1})^\prime$ and $(E_1 \otimes \cdots \otimes E_n\otimes F,\pi_{n+1})^\prime$. Since the ideal $\mathcal{L}$ is $\pi$-represented, we have the  following chain of canonical isomorphisms:
$$
{\cal M}(E_1, \ldots, E_n;F^\prime) \longleftrightarrow (E_1 \otimes \cdots \otimes E_n\otimes F, \beta_{n+1})'$$
$$\longleftrightarrow (E_1 \otimes \cdots \otimes E_n\otimes F, \pi_{n+1})' \longleftrightarrow {\cal L}(E_1, \ldots,
E_n;F^\prime).%
$$
It follows that the identity operator is an algebraic isomorphism between the spaces ${\cal M}(E_1, \ldots, E_n;F^\prime)$ and ${\cal L}(E_1, \ldots,
E_n;F^\prime)$, which is a contradiction.
\end{proof}

The next multilinear generalization of the ideal of absolutely
summing linear operators was introduced in \cite{port}:

\begin{definition}\rm
\textrm{Given $p\geq q\geq1$, a multilinear mapping $T\in\mathcal{L}%
(E_{1},\ldots,E_{n};F)$ is said to be \textit{strongly multiple $(p,q)$%
-summing} if there exists $C\geq0$ such that
\[
\left(  \sum\limits_{j_{1},\ldots,j_{n}=1}^{m}\| T(x_{j_{1}}^{(1)}%
,\ldots,x_{j_{n}}^{(n)})\|^{p}\right)  ^{1/p}\leq C\left(  \underset
{}{\underset{\phi\in B_{\mathcal{L}(E_{1},\ldots,E_{n})}}{\sup}}%
\sum\limits_{j_{1},\ldots,j_{n}=1}^{m}\mid\phi(x_{j_{1}}^{(1)},\ldots
,x_{j_{n}}^{(n)})\mid^{q}\right)  ^{1/q}%
\]
for every $m\in\mathbb{N}$, $x_{j_{l}}^{(l)}\in E_{l}$ with
$l=1,\ldots,n$ and $j_{l}=1,\ldots,m.$ The space of all strongly
multiple $p$-summing $n$-linear mappings from
$E_{1}\times\cdots\times E_{n}$ to $F$ will be denoted by
$\mathcal{L}_{sm(p,q)}(E_{1},\ldots,E_{n};F)$. The infimum of the
constants $C$ for which the inequality always holds defines a
complete norm $\Vert \cdot\Vert_{sm(p,q)}$ on
$\mathcal{L}_{sm(p,q)}(E_{1},...,E_{n};F)$. When $p=q$ we shortly
write $\mathcal{L}_{sm,p}(E_{1},\ldots,E_{n};F)$ and
$\Vert\cdot\Vert_{s{f},p}$. }
\end{definition}

\begin{proposition}
\label{not} The multi-ideal $\mathcal{L}_{sm,p}$, $1 \leq p < +
\infty$, of strongly multiple
$p$-summing multilinear mappings is represented by a tensor norm but not by a smooth tensor
norm.
\end{proposition}

\begin{proof} Following the lines of \cite[Proposici\'on
4.37]{david} it is not difficult to prove that ${\cal L}_{sm,p}$ is maximal. So it follows from
\cite[Theorem 4.5]{hf} that ${\cal L}_{sm,p}$ is represented by a
tensor norm.

 Let us see that ${\cal L}_{sm,p}$ cannot be represented by a smooth
tensor norm. It is clear that $ {\cal L}_{sm,p}(E_1, \ldots,
E_n,F;\mathbb{K}) = {\cal L}(E_1, \ldots, E_n,F;\mathbb{K}) $ for
every integer $n$ and every Banach spaces $E_1, \ldots, E_n$ and
$F$. On the other hand, assuming that ${\cal L}_{sm,p}(^nE,;F^\prime) = {\cal L}(^nE;F^\prime)$ for every
integer $n$ and every Banach spaces $E$ and $F$,
by \cite[Proposition 5.2(iii)]{port} we would have ${\cal
L}(E;F^\prime) = \Pi_p(E;F^\prime)$ for every $E$ and $F$. This is
absurd because for every infinite-dimensional Banach space $E$ the
canonical injection $E \hookrightarrow E^{''} = (E')'$ fails to be
$p$-summing. Hence ${\cal L}(^nE;F^\prime) \neq {\cal L}_{sm,p}(^nE;F^\prime)$ for
some $n$ and some Banach spaces $E$ and $F$. By Proposition \ref{p:111} it
follows that ${\cal L}_{sm,p}$ cannot be represented by a smooth
tensor norm.
\end{proof}

\begin{corollary}\label{coro} For $\mathcal{M} = \mathcal{L}_{sm,p}$ we have
$\mathcal{M} \ne\mathcal{L}_{\beta^{\mathcal{M}}}$.
\end{corollary}

\begin{proof} It is easy to check that $\mathcal{L}_{sm,p}$ has property [B]. Indeed, it is enough to combine the definition of
${\cal L}_{sm,p}$ with the well-known fact that the space of
$(n+1)$-linear forms $\mathcal{L}(E_{1}, \ldots, E_{n},
\mathbb{K};\mathbb{K})$ is isometrically isometric to the space of
$n$-linear forms $\mathcal{L}(E_{1}, \ldots, E_{n};\mathbb{K})$ via
the obvious correspondence. Assuming that
$\mathcal{M} = \mathcal{L}_{\beta^{\mathcal{M}}}$, by Proposition \ref{max} $\mathcal{L}_{sm,p}$ would be represented by a smooth tensor
norm; but this is not true by
Proposition \ref{not}.
\end{proof}

\begin{remark}\rm For the sake of completeness, let us construct the tensor norm that represents the
multi-ideal ${\cal L}_{sm,p}$: Given normed spaces
$E_{1},\ldots,E_{n},F$ and $p\geq1$, define\\

$ \beta_{p}(u):=$
\[\inf\left[  \sum_{m=1}^{M}\|(b_{m,j^{1}_{m},\ldots
,j^{n}_{m}})_{j^{1}_{m},\ldots,j^{n}_{m}=1}^{I^{1}_{m},\ldots,I^{n}_{m}%
}\|_{q} \left(  {\underset{\phi\in
B_{\mathcal{L}(E_{1},\ldots,E_{n})}}
{\sup}}\sum\limits_{j^{1}_{m},\ldots,j^{n}_{m}=1}^{I^{1}_{m},\ldots,I^{n}_{m}%
}\mid\phi(x^{(1)}_{m,j^{1}_{m}},\ldots,x^{(n)}_{m,j^{n}_{m}})\mid^{p}\right)
^{1/p}\right]
\]
where the infimum is taken over the set of all representations of
the tensor $u\in E_{1}\otimes\cdots\otimes E_{n}\otimes F$ of the
form
$$u=\sum_{m=1}^{M}\sum\limits_{j^{1}_{m},\ldots,j^{n}_{m}=1}^{I^{1}_{m}%
,\ldots,I^{n}_{m}} x^{(1)}_{m,j^{1}_{m}}\otimes\cdots\otimes x^{(n)}_{m,j^{n}%
_{m}}\otimes b_{m,j^{1}_{m},\ldots,j^{n}_{m}} $$
with $M\in\mathbb{N}$, $x^{(l)}_{m,j^{l}_{m}}\in E_{l}$,
$l=1,\ldots,n$, $b_{m,j^{1}_{m},\ldots,j^{n}_{m}}\in F$,
$j^{l}_{m}=1,\ldots,I^{l}_{m}$, and $q\geq1$ with
$\frac{1}{p}+\frac{1}{q}=1$. Following the lines of \cite[Teorema 4.38]{david} one can prove that
$\beta_p$ is a tensor norm and that ${\cal L}_{sm,p}$ is
$\beta_p$-represented. From Proposition \ref{not} it follows that
$\beta_p$ is not smooth.\label{r:12}
\end{remark}


The ideal $\mathcal{L}_{m,p}$ of multiple $p$-summing multilinear mappings, introduced by Matos \cite{mapre} and, independently, by Bombal, P\'erez-Garc\'ia and Villanueva \cite{bpgv}, has played a central role in the theory of multi-ideals, providing even unexpected applications (see Remark \ref{remark}(a)).

\begin{proposition}
For $1 \leq p \leq2$, the ideal $\mathcal{L}_{m,p}$ of multiple
$p$-summing multilinear mappings is represented by a tensor norm
but not by a smooth tensor norm.
\end{proposition}

\begin{proof} That $\mathcal{L}_{m,p}$ is represented by a tensor norm is proved in \cite[Proposici\'on 4.39]{david}. From
$\mathcal{L}_{m,p}(^{2}\ell_{1};\mathbb{K}) = \mathcal{L}(^{2} \ell
_{1};\mathbb{K})$ \cite[Teorema 5.23]{david} and $\Pi_{p}(\ell_{1};
\ell_{\infty})= \mathcal{L}_{m,p}(\ell_{1}; \ell_{\infty}) \neq\mathcal{L}%
(\ell_{1};\ell_{\infty})$ (obvious), the non-representability by a smooth tensor norm follows from Proposition \ref{p:111}.
\end{proof}

\begin{remark}\rm \label{remark}(a) As proved by Defant and P\'erez-Garc\'ia \cite{dpg}, the tensor norm that represents the ideal $\mathcal{L}_{m,p}$ of multiple $p$-summing multilinear mappings is the first example of a tensor norm that preserves unconditionality for ${\cal L}_p$-spaces.\\
\noindent (b) Everything we proved for the multi-ideal
$\mathcal{L}_{sm,p}$ of strongly multiple $p$-summing multilinear mappings can
be proved, \textit{mutatis mutandis}, for the multi-ideal $\mathcal{L}_{ss,p}$
of strongly $p$-summing multilinear mappings introduced in \cite{dimant}.
\end{remark}


We finish this section with another hint that the representation of a (vector-valued) multi-ideal by a smooth tensor norm is not commonplace.

\begin{proposition}\label{propos}
{ The multi-ideal ${\cal L}_{\varepsilon}$ is not represented by a
smooth tensor norm.}
\end{proposition}

\begin{proof}Assume that $\mathcal{L}_\varepsilon$ is represented by a smooth tensor norm $\beta=(\beta_n)_{n=1}^\infty$. Fix a positive integer $n$ and Banach spaces $E_1,\ldots,E_n$. Since $\mathcal{L}_\varepsilon$ is $\beta$-represented, the mapping

$$\varphi:(\mathcal{L}_\varepsilon(E_1,\ldots, E_n;F'),\parallel.\parallel_{\mathcal{L}_\varepsilon})\longrightarrow
(E_1\otimes\cdots\otimes E_n\otimes F,\beta_{n+1})'$$
$$T\longrightarrow\varphi(T)(x_1\otimes\cdots\otimes x_n\otimes y)=T(x_1,\ldots,x_n)(y)$$
is an isometric isomorphism for every $F$. In particular, taking $F=\mathbb{K}$, we obtain the following isometric isomorphisms:
 $$(\mathcal{L}_\varepsilon(E_1,\ldots, E_n;\mathbb{K}'),\parallel.\parallel_{\mathcal{L}_\varepsilon})\cong(E_1\otimes\cdots\otimes E_n\otimes \mathbb{K},\beta_{n+1})'\cong(E_1\otimes\cdots\otimes E_n,\beta_{n})'.$$
 In Corollary \ref{cor} we shall prove that following spaces are also isometric isomorphic, with the same canonical correspondences:
$$(\mathcal{L}_\varepsilon(E_1,\ldots, E_n;\mathbb{K}'),\parallel.\parallel_{\mathcal{L}_\varepsilon})\cong(E_1\otimes\cdots\otimes E_n\otimes \mathbb{K},\varepsilon_{n+1})'\cong(E_1\otimes\cdots\otimes E_n,\varepsilon_{n})'.$$
\noindent Hence the identity mapping $(E_1\otimes\cdots\otimes E_n,\varepsilon_n)'\longrightarrow(E_1\otimes\cdots\otimes E_n,\beta_n)'$ is an isometric isomorphism. Calling on Hahn-Banach once again we get that $\varepsilon_n$ and $\beta_n$ coincide on $E_1\otimes\cdots\otimes E_n$. It follows that the ideal $\mathcal{L}_\varepsilon$ is $\varepsilon$-represented, which contradicts Proposition \ref{injec}.
\end{proof}

\section{Scalar-valued case}

The aim of this section is show that smooth tensor norms are more suitable to represent
ideals of multilinear forms.

Given a tensor norm $\beta$, we write
$\mathcal{L}_{\beta}^{\mathbb{K}}:=(\mathcal{L}_{\beta})^{\mathbb{K}}%
$.\newline\indent We shall write ${\cal M}^{\mathbb{K}}
\stackrel{1}{=} {\cal L}_{\beta}^{\mathbb{K}}$ if for every $n$ and
every $E_1, \ldots, E_n$, the linearization operator $\Phi = \Phi(n,
E_1, \ldots, E_n)$: $$A \in {\cal M}(E_1, \ldots, E_n; \mathbb{K})
\mapsto \Phi(A) := A_L,$$ is an isometric isomorphism from ${\cal
M}(E_1, \ldots, E_n; \mathbb{K})$ onto $(E_1 \otimes\cdots \otimes
E_n, \beta_{n})^\prime$.

\begin{theorem}
\label{th}Let $\mathcal{M}$ be a multi-ideal and $\beta= (\beta_{n}%
)_{n=1}^{\infty}$ be a smooth tensor norm. Then the ideal of
multilinear forms $\mathcal{M}^{\mathbb{K}}$ is $\beta$-represented
if and only if ${\cal M}^{\mathbb{K}}
\stackrel{1}{=} {\cal L}_{\beta}^{\mathbb{K}}$.
\end{theorem}

\begin{proof} Let $n \in \mathbb{N}$ and $E_1, \ldots, E_n$ be Banach spaces. We shall say
that $\Phi$ is well defined, isometric and onto if $\Phi(A) = A_L
\in (E_1 \otimes\cdots \otimes E_n, \beta_{n})^\prime$ {for every $A
\in {\cal M}(E_1, \ldots, E_n; \mathbb{K})$} and $\Phi \colon {\cal
M}(E_1, \ldots, E_n; \mathbb{K}) \longrightarrow (E_1 \otimes\cdots
\otimes E_n, \beta_{n})^\prime$ is {isometric and} surjective.
Assume that $\Phi$ is well defined, {isometric} and onto. It is
clear that $\Phi$ is linear and injective, {so $\Phi$ is well
defined, isometric and onto if and only if $\Phi$ is an isometric
isomorphism from ${\cal M}(E_1, \ldots, E_n; \mathbb{K})$ onto $(E_1
\otimes\cdots \otimes E_n,
\beta_{n})^\prime$.}\\
\indent We continue assuming that $\Phi$ is well defined,{isometric}
and onto. Let $h \colon \mathbb{K}' \longrightarrow \mathbb{K}$ be
the isometric isomorphism given by $h(f) = f(1)$ for every $f \in
\mathbb{K}'$. It is clear that the linear mapping
$$\xi \colon {\cal M}(E_1, \ldots, E_n;\mathbb{K}^\prime) \longrightarrow {\cal M}(E_1, \ldots, E_n; \mathbb{K})~,~\xi(A) = h \circ A,  $$
is an isometric isomorphism as well. Considering the chain
$${\cal M}(E_1, \ldots, E_n;\mathbb{K}^\prime) \stackrel{\xi}{\longrightarrow} {\cal M}(E_1, \ldots, E_n; \mathbb{K})
\stackrel{\Phi}{\longrightarrow} (E_1 \otimes\cdots \otimes E_n, \beta_{n})^\prime$$
$$\stackrel{\psi^*}{\longrightarrow}
(E_1 \otimes\cdots \otimes E_n \otimes \mathbb{K},
\beta_{n+1})^\prime,$$ where $\psi^*$ is the adjoint of the linear
operator $\psi$ of the definition of smooth tensor norm, it is not
difficult to check that $\varphi = \psi^* \circ \Phi \circ \xi$
(hence $\Phi = (\psi^*)^{-1} \circ \varphi \circ \xi^{-1})$, where
$\varphi = \varphi(n, E_1, \ldots, E_n)$ is the operator of
Definition \ref{definition}. \\
\indent Hence ${\cal M}^{\mathbb{K}} \stackrel{1}{=} {\cal L}_{\beta}^{\mathbb{K}}$ if and only if $\Phi(n,
E_1, \ldots, E_n)$ is an isometric isomorphism for every $n, E_1, \ldots, E_n$ if and only if $\Phi(n, E_1,
\ldots, E_n)$ is well defined, isometric and onto for every $n, E_1, \ldots, E_n$ if and only
if $\varphi(n, E_1, \ldots, E_n)$ is an isometric isomorphism for every $n, E_1, \ldots, E_n$ if and only if
${\cal M}^{\mathbb{K}}$ is $\beta$-represented.
\end{proof}

As we saw before (see, e.g., Proposition \ref{injec} and Corollary \ref{coro}), the theorem above cannot be generalized to vector-valued multi-ideals.

\begin{corollary}
\label{cor} Let $\beta$ be a smooth tensor norm. Then the ideal of
multilinear forms $\mathcal{L}_{\beta}^{\mathbb{K}}$ is
{$\beta$-represented and thus is} represented by a smooth tensor
norm.
\end{corollary}

\begin{example}\label{scalar}\rm The ideal of multilinear forms ${\cal L}_{\varepsilon}^\mathbb{K}$ is $\varepsilon$-represented, hence represented by a smooth tensor norm.
\end{example}

\begin{corollary}
{Let ${\cal M}$ be a multi-ideal. If ${\cal M}^{\mathbb{K}}$ is
represented by a smooth tensor norm then ${\cal M}^{\mathbb{K}}$
contains the integral multilinear forms.}
\end{corollary}

\begin{proof}{By Theorem \ref{th} we have that ${\cal M}^{\mathbb{K}} = {\cal L}_{\beta}^{\mathbb{K}}$, where
$\beta$ is a smooth tensor norm. Since $\varepsilon \leq \beta$ because $\beta$ is a tensor norm, it follows
that $ {\cal L}_{\cal I}^{\mathbb{K}} = {\cal L}_{\varepsilon}^{\mathbb{K}} \subseteq {\cal
L}_{\beta}^{\mathbb{K}} = {\cal M}^{\mathbb{K}}$.}
\end{proof}

Next we see that sometimes we can construct explicitly the smooth tensor norm
that represents an ideal of multilinear forms.

\begin{definition}\rm
\textrm{Let $p \geq1$. An $n$-linear mapping $A \in\mathcal{L}(E_{1},\ldots,
E_{n};F)$ is $p-$\textit{semi-integral}, in symbols $A \in\mathcal{L}%
_{si,p}(E_{1},\ldots,E_{n};F)$, if there exist $C\geq0$ and a regular
probability measure $\mu$ on the Borel $\sigma-$algebra of $B_{E_{1}%
^{^{\prime}}}\times\cdots\times$ $B_{E_{n}^{^{\prime}}}$ endowed with the
product of the weak star topologies $\sigma(E_{l}^{\prime},E_{l}),$
$l=1,\ldots,n,$ such that
\[
\| A(x_{1},\ldots,x_{n})\|\leq C\left(  \int_{B_{E_{1}^{\prime}}\times
\cdots\times B_{E_{n}^{\prime}}}|\varphi_{1}(x_{1})\cdots\varphi_{n}%
(x_{n})|^{p}d\mu(\varphi_{1},\ldots,\varphi_{n})\right)  ^{1/p}%
\]
for every $x_{j}\in E_{j}$, $j=1,\ldots,n$. The infimum of the constants $C$
defines a norm $\|\cdot\| _{si,p}$ on $\mathcal{L}_{si,p}(E_{1},\ldots
,E_{n};F)$. It is well known that $\mathcal{L}_{si,p}$ is a multi-ideal (see
\cite{CD}). }
\end{definition}

\begin{proposition}
\label{semi} For $p \geq1$, the ideal of $p$-semi-integral multilinear forms
$\mathcal{L}_{si,p}^{\mathbb{K}}$ is represented by a smooth tensor norm.
\end{proposition}

\begin{proof} Given $u\in
E_{1}\otimes\cdots\otimes E_{n}$, define
\[
\sigma_{p}^{n}(u):=\inf\Vert(\lambda_{j})_{j=1}^{m}\Vert_{q}\left(
\sup_{\varphi_{l}\in B_{E_{l}^{^{\prime}}}}\sum_{j=1}^{m}|\varphi_{1}%
(x_{1,j})\cdots\varphi_{n}(x_{n,j})|^{p}\right)  ^{1/p}%
\]
where the infimum is taken over all representations of the form $u=\sum
\limits_{j=1}^{m}\lambda_{j}x_{1,j}\otimes\cdots\otimes x_{n,j}$, where
$m\in\mathbb{N}$, $x_{l,j}\in E_{l}$, $l=1,\ldots,n$, $\lambda_{j}%
\in\mathbb{K}$, $j=1,\ldots,m$, and $\frac{1}{p}+\frac{1}{q}=1$. Standard
techniques show that $\sigma_{p}=(\sigma_{p}^{n})_{n=1}^{\infty}$ is a tensor
norm (hardwork!). Let us show that $\sigma_{p}$ is smooth. Given $u=\sum
_{j=1}^{m}\lambda_{j}x_{1,j}\otimes\cdots\otimes x_{n,j}\otimes b_{j}$ in
$E_{1}\otimes\cdots\otimes E_{n}\otimes\mathbb{K} $, we have
\begin{align*}
({\sigma}_{p}^{n+1}(u))^{p}  & \leq\Vert(\lambda_{j})_{j=1}^{m}\Vert_{q}%
^{p}\sup_{\varphi_{l}\in B_{E_{l}^{^{\prime}}},\varphi\in B_{\mathbb{K}%
^{^{\prime}}}}\displaystyle\sum_{j=1}^{m}\mid\varphi_{1}(x_{1,j})\cdots
\varphi_{n}(x_{n,j})\varphi(b_{j})\mid^{p}\\
& =\Vert(\lambda_{j})_{j=1}^{m}\Vert_{q}^{p}\sup_{\varphi_{l}\in
B_{E_{l}^{^{\prime}}},\varphi\in B_{\mathbb{K}^{^{\prime}}}}\displaystyle\sum
\limits_{j=1}^{m}\mid\varphi_{1}(b_{j}x_{1,j})\cdots\varphi_{n}(x_{n,j}%
)\varphi(1)\mid^{p}\\
& \leq\Vert(\lambda_{j})_{j=1}^{m}\Vert_{q}^{p}\sup_{\varphi_{l}\in
B_{E_{l}^{^{\prime}}},\varphi\in B_{\mathbb{K}^{^{\prime}}}}\displaystyle\sum
\limits_{j=1}^{m}\mid\varphi_{1}(b_{j}x_{1,j})\cdots\varphi_{n}(x_{n,j}%
)\mid^{p}\Vert\varphi\Vert^{p}\mid1\mid^{p}\\
& \leq\Vert(\lambda_{j})_{j=1}^{m}\Vert_{q}^{p}\sup_{\varphi_{l}\in
B_{E_{l}^{^{\prime}}}}\displaystyle\sum\limits_{j=1}^{m}\mid\varphi_{1}%
(b_{j}x_{1,j})\cdots\varphi_{n}(x_{n,j})\mid^{p}.
\end{align*}
Since $\sum_{j=1}^{m}\lambda_{j}b_{j}x_{1,j}\otimes\cdots\otimes x_{n,j}$ is a
representation of $\psi(u)\in E_{1}\otimes\cdots\otimes E_{n}$, it follows
that ${\sigma}_{p}^{n+1}(u)\leq{\sigma}_{p}^{n}(\psi(u))$. A similar
computation shows that ${\sigma}_{p}^{n}(\psi(u))\leq{\sigma}_{p}^{n+1}(u)$,
completing the proof of the smoothness of $\sigma_{p}$.\\
\indent The proof that $\mathcal{L}_{si,p}^{\mathbb{K}}$ is $\sigma_{p}$-represented is a combination of the arguments of the proofs of \cite[Theorem 4.8]{AlencarMatos} and \cite[Theorem
1]{CD}. Let $E_1, \ldots, E_n$ and $f \in(E_{1}\otimes\cdots\otimes E_{n},\sigma_{p}^n)^{\prime}$ be given.
Consider the $n$-linear form
$A$ on $E_{1}\times\cdots\times E_{n}$ given by $A(x_{1}%
,\ldots,x_{n})=f(x_{1}\otimes\cdots\otimes x_{n})$. For appropriate $\lambda_{j}\in\mathbb{K}$ with
$|\lambda_{j}|=1$,$j=1,\ldots,m$, we have that
$$ \sum_{j=1}^{m}|
A(x_{1,j},\ldots,x_{n,j})|^{p}= \sum_{j=1}^{m}\left| | f(x_{1,j}\otimes\cdots\otimes
x_{n,j})|^{p-1}f(x_{1,j}\otimes \cdots\otimes x_{n,j})\right| \hspace*{30em}$$
$$ =\left|
\sum_{j=1}^{m}|f(x_{1,j}\otimes\cdots\otimes x_{n,j})|^{p-1}\lambda_{j}f(x_{1,j}\otimes\cdots\otimes
x_{n,j})\right| \hspace*{30em} $$
$$=\left|f\left(  \sum_{j=1}%
^{m}\lambda_{j}|f(x_{1,j}\otimes\cdots\otimes x_{n,j})|^{p-1}%
x_{1,j}\otimes\cdots\otimes x_{n,j}\right)  \right|\hspace*{30em}$$
$$ \leq\|f\|_{(E_1\otimes\cdots\otimes
E_n,\sigma_p^n)'}\sigma_{p}^n \left(  \sum _{j=1}^{m}\lambda_{j}|f(x_{1,j}\otimes\cdots\otimes x_{n,j})|
^{p-1}x_{1,j}\otimes\cdots\otimes x_{n,j}\right)\hspace*{30em}$$
$$ \leq\|f\|_{(E_1\otimes\cdots\otimes
E_n,\sigma_p^n)'}\left\|  \left(  \lambda_{j}|
f(x_{1,j}\otimes\cdots\otimes x_{n,j})|^{p-1}\right)  _{j=1}%
^{m}\right\|  _{q}\left( \sup_{\varphi_{l}\in B_{E_{l}^{^{\prime}}}}\sum\limits_{j=1}^{m}|\varphi_{1}
(x_{1,j})\cdots\varphi_{n} (x_{n,j})|^{p}\right)  ^{1/p}\hspace*{30em}$$
$$
=\displaystyle\|f\|_{(E_1\otimes\cdots\otimes E_n,\sigma_p^n)'}\left(  \sum_{j=1}%
^{m}|f(x_{1,j}\otimes\cdots\otimes x_{n,j})|^{p}\right) ^{1/q}\left(  \sup_{\varphi_{l}\in
B_{E_{l}^{^{\prime}}}}\sum\limits_{j=1}^{m}|\varphi_{1} (x_{1,j}%
)\cdots\varphi_{n} (x_{n,j})|^{p}\right)  ^{1/p}.\hspace*{30em}$$
Hence we have that
$$\left(  \sum_{j=1}^{m}| A(x_{1,j},\ldots,x_{n,j})|^{p}\right) ^{1/p}\leq\|f \|_{(E_1\otimes\cdots\otimes E_n,
\sigma_p^n)'}\left( \sup_{\varphi_{l}\in
B_{E_{l}^{^{\prime}}}}\sum\limits_{j=1}^{m}|\varphi_{1}
(x_{1,j})\cdots\varphi_{n} (x_{n,j})|^{p}\right) ^{1/p},$$ which
shows that $A\in \mathcal{L}_{si,p}(E_{1},\ldots,E_{n})$ and $$\|
A\|_{si,p}\leq\|f\|_{(E_1\otimes\cdots\otimes E_n,\sigma_p^n)'} {=
\|\Phi(A)\|_{(E_1\otimes\cdots\otimes E_n,\sigma_p^n)'}}.$${ In
particular $\Phi$ is surjective. To show the reverse inequality, take
$A\in {\mathcal L}_{si,p}(E_{1},\ldots,E_{n};\mathbb{K})$ and
let $u = \sum_{j=1}^{m}\lambda_{j}x_{1,j}\otimes\cdots\otimes
x_{n,j} \in E_{1}\otimes\cdots\otimes E_{n}$, where
$\lambda_{j}\in\mathbb{K}$, $x_{l,j}\in E_{l}$, $l=1,\ldots,n$,
$j=1,\ldots,m$. From \cite[Theorem 1]{CD} it follows that
\begin{eqnarray*}|\Phi(A)(u)|^{p}&=&\left|
\sum_{j=1}^{m}\lambda_{j}A(x_{1,j},\ldots,x_{n,j})\right|  ^{p}\leq \|(\lambda_{j})_{j=1}^{m}\|^{p}_{q}
\sum\limits_{j=1}^{m}| A(x_{1,j},\ldots,x_{n,j})|^{p}\\
&\leq&\| (\lambda_{j})_{j=1}^{m}\|^{p}_{q}\| A\|^{p}_{si,p} \sup_{\varphi_{l}\in
B_{E_{l}^{^{\prime}}}}\sum\limits_{j=1}^{m}|\varphi_{1} (x_{1,j})\cdots\varphi_{n} (x_{n,j}%
)|^{p},
\end{eqnarray*}
so $|\Phi(A)(u)| \leq\| A\|_{si,p}\cdot\sigma_{p}(u)$. Since $u$ is arbitrary it follows that $\|\Phi(A)
\|_{(E_1\otimes\cdots\otimes E_n,\sigma_p)'}\leq\| A\|_{si,p}$.}
\end{proof}

Summing up the information:

\medskip

\noindent (i) The multi-ideal ${\cal L}_{\varepsilon}$ is not represented by a
smooth tensor norm (Proposition \ref{propos}), whereas ${\cal L}_{\varepsilon}^\mathbb{K}$ is represented by the smooth tensor norm $\varepsilon$ (Example \ref{scalar}),\\
(ii) The ideal $\mathcal{L}_{sm,p}$ of strongly multiple $p$-summing multilinear mappings is represented by a tensor norm but not by a smooth
tensor norm (Proposition \ref{not}) and, similarly to what we did in Proposition \ref{semi}, it can be proved that $\mathcal{L}_{sm,p}%
^{\mathbb{K}}$ is represented by a smooth tensor norm;\\
(iii) The ideal $\mathcal{L}_{si,p}^{\mathbb{K}}$ of $p$-semi-integral multilinear forms is represented by a smooth tensor norm (Proposition \ref{semi}), the multi-ideal $\mathcal{L}_{si,p}$ of $p$-semi-integral multilinear mappings is represented by a tensor norm (\cite[Proposition 2]{note}) but not by a smooth tensor norm (hardwork!),

\medskip

\noindent it seems that the typical behavior of a maximal multi-ideal ${\cal M} \neq {\cal L}$ is that ${\cal M}$ is not represented by a smooth tensor norm whereas ${\cal M}^\mathbb{K}$ is represented by a smooth tensor norm.

We finish the paper showing how the representation of the ideal of multilinear functionals ${\cal M}^{\mathbb{K}}$ by
a smooth tensor norm provides information about the representation of certain vector-valued components of ${\cal M}$. Given Banach spaces
$E_1, \ldots, E_n,F$ and a tensor norm $\beta=(\beta_n)_{n=1}^\infty$, we shall say that ${\cal M}(E_1, \ldots, E_n;F')$ is $\beta$-represented if
$\mathcal{M}(E_{1}, \ldots, E_{n};F^{\prime})$ and $(E_{1} \otimes\cdots\otimes E_{n}\otimes F, \beta_{n+1})^{\prime}$ are isometrically isomorphic via the canonical mapping of Definition \ref{definition}.

\begin{proposition}\label{propo} Let $\mathcal{M}$ be a multi-ideal such that $\mathcal{M}^{\mathbb{K}}$ is represented by a smooth tensor norm $\beta$. Then ${\cal M}(E_1, \ldots, E_n;F')$ is $\beta$-represented whenever the spaces $\mathcal{M}(E_1,\ldots,E_n;F')$ and $\mathcal{M}(E_1,\ldots,E_n,F;{\mathbb{K}})$ are canonically isometric isomophic. 
\end{proposition}

\begin{proof} Consider the following chain of canonical mappings:
\begin{eqnarray*}\mathcal{M}(E_1,\ldots,E_n;F')& \stackrel{\rm I}= &\mathcal{M}(E_1,\ldots,E_n,F;\mathbb{K})\\
& \stackrel{\rm II}= & \mathcal{M}(E_1,\ldots,E_n,F;\mathbb{K}')\\
& \stackrel{\rm III}= & (E_1\otimes\cdots \otimes E_n\otimes F\otimes\mathbb{K}, \beta_{n+2})' \\
& \stackrel{\rm IV}= &(E_1\otimes\cdots \otimes E_n\otimes F, \beta_{n+1})'.
\end{eqnarray*}
I is an isometric isomorphism by assumption, II because the canonical mapping
$$h \colon \mathbb{K} \longrightarrow \mathbb{K}'~,~h(\lambda)(\alpha) = \lambda \cdot \alpha $$
is an isometric isomorphism, III because $\mathcal{M}^{\mathbb{K}}$ is $\beta$-represented and IV because the tensor norm $\beta$ is smooth. Routine computations show that the composition of all these mappings ends up in the canonical mapping between $\mathcal{M}(E_1,\ldots,E_n;F')$ and $(E_1\otimes\cdots \otimes E_n\otimes F, \beta_{n+1})'$.
\end{proof}

\vspace*{1em} \noindent[Geraldo Botelho] Faculdade de Matem\'atica,
Universidade Federal de Uberl\^andia, 38.400-902 - Uberl\^andia, Brazil,
e-mail: botelho@ufu.br.

\medskip

\noindent[Erhan \c Caliskan] Y\i ld\i z Tekn\'{\i}k \"{U}n\'{\i}vers\'{\i}tes\'{\i},
Fen-Edeb\'{\i}yat Fak\"{u}ltes\'{\i}, Matemat\'{\i}k B\"{o}l\"{u}m\"{u},
Davutpa\c sa Kamp\"{u}s\"{u}, 34210 Esenler, \.{I}stanbul, T\"{u}rk\'{\i}ye, e-mail: \newline caliskan@yildiz.edu.tr.

\medskip

\noindent\lbrack Daniel Pellegrino] Departamento de Matem\'{a}tica,
Universidade Federal da Para\'{\i}ba, 58.051-900 - Jo\~{a}o Pessoa, Brazil,
e-mail: dmpellegrino@gmail.com.

\end{document}